\documentclass[12pt]{amsart}
\usepackage[utf8]{inputenc}
\usepackage[margin=1.15in]{geometry}
\usepackage{amsmath, amssymb, tikz, tikz-cd, amsthm, hyperref, enumitem, stmaryrd, bm, mathdots}

\usepackage[alphabetic]{amsrefs}

\title{Rational approximation for Hitchin representations}
\author{Jacques Audibert}
\author{Michael Zshornack}

\address{Max Planck Institute for Mathematics, Vivatsgasse 7, 53111 Bonn, Germany.}
\email{\href{mailto:audibert.j@outlook.fr}{audibert.j@outlook.fr}}
\address{Department of Mathematics, University of California Santa Barbara, Santa Barbara, CA 93106, USA.}
\email{\href{mailto:zshornack@math.ucsb.edu}{zshornack@math.ucsb.edu}}

\newtheorem{theorem}{Theorem}[section]
\newtheorem{corollary}[theorem]{Corollary}
\newtheorem{lemma}[theorem]{Lemma}

\newtheorem{question}{Question}

\theoremstyle{remark}
\newtheorem*{remark}{Remark}

\theoremstyle{definition}
\newtheorem{definition}[theorem]{Definition}

\DeclareMathOperator{\SL}{SL}
\DeclareMathOperator{\Sp}{Sp}
\DeclareMathOperator{\SO}{SO}
\DeclareMathOperator{\Hom}{Hom}
\DeclareMathOperator{\Hit}{\mathcal{H}}
\DeclareMathOperator{\Tr}{Tr}

\newcommand{\C}{\mathbf{C}}
\newcommand{\R}{\mathbf{R}}
\newcommand{\Z}{\mathbf{Z}}
\newcommand{\Q}{\mathbf{Q}}

\begin{document}

\begin{abstract}
    A consequence of Rapinchuk et al. is that for $S$ a closed surface of genus $g\geq 2$, the set of Hitchin representations of $\pi_1(S)$ with image in $\SL(n,\Q)$ is dense in the Hitchin component. We give a dynamical proof of this fact provided that $g\geq 3$. Moreover, we extend it to some other $\Q$-groups such as $\Sp(2k,\Q)$ and $\mathrm{G}_2(\Q)$, where the results are new.
\end{abstract}

\maketitle

\section{Introduction}

Let $S$ be a closed, orientable surface of genus at least $2$ and let 
\[
\Hit_n(S)\subset\Hom(\pi_1(S),\SL(n,\R))/\SL(n,\R)
\]
denote the Hitchin component of $S$.\footnote{For $n$ even, there are two such components. We pick one of them.} It is a connected component that consists only of discrete and faithful representations. Representations in these components have many interesting dynamical and geometric properties, the study of which comprises a very active field of modern research (e.g. see \cite{canary2023hitchin}). Their arithmetic properties also provide a powerful tool in the study of lattices in higher rank. See for instance \cite{long2011zariski} and \cite{LT} where the authors use Hitchin representations to construct Zariski-dense surface subgroups of $\SL(n,\Z)$ for every odd $n$.

More generally, the underlying rational structure of the Hitchin component provides means of understanding surface subgroups of $\SL(n,\Q)$. Such subgroups also have a number of interesting properties, for one, they satisfy analogs of strong approximation \cite{matthewsetal}. In addition, while not necessarily contained in lattices of $\SL(n,\R)$, these subgroups are contained in lattices of products of Lie groups and $p$-adic Lie groups. As such, they admit actions on spaces constructed from the symmetric spaces and Bruhat--Tits buildings associated to the groups in this product. Numerous open questions surround the properties of such actions, see \cite{fisheretal} or \cite{brodyetal}.

In this note we investigate one aspect of the Hitchin component's rational structure. Let $\Hit_n(S)_{\Q}$ denote the conjugacy classes of representations in $\Hit_n(S)$ whose image is conjugate to a subgroup of $\SL(n,\Q)$. Representations in $\Hit_n(S)_{\Q}$ provide surface subgroups of $\SL(n,\Q)$ and our main result is that they are abundant, in the sense of the following theorem.

\begin{theorem}
\label{maintheorem}
When the genus of $S$ is at least $3$, $\Hit_n(S)_{\Q}$ is dense in $\Hit_n(S)$. 
\end{theorem}

Theorem \ref{maintheorem} is also a consequence of the $\Q$-rationality of these varieties, which was originally proven by \cite{Rapinchuk_Representation}. Nonetheless, our methods recover this fact using different techniques. Moreover, generalizations to representations into other Lie groups not considered in their original work are discussed in \S \ref{section:generalizations}.

From now on, we make the standing assumption that the genus of $S$ is at least $3$. Some discussion on what is missing in the genus $2$ case, as well as a discussion of other generalizations of Theorem \ref{maintheorem}, is left to Section \ref{section:generalizations}.

When $n=2$, $\Hit_2(S)$ is the Teichm\"uller space of $S$ and Theorem \ref{maintheorem} recovers a particular case of a classical result of Takeuchi \cite{takeuchi}. In this case, the images of the representations studied are always cocompact lattices in $\SL(2,\R)$, whereas they are always of infinite covolume when $n>2$, thus our proof uses different methods.

While Theorem \ref{maintheorem} concerns conjugacy classes of representations, the same statement holds at the level of individual representations. Let 
\[
\widetilde{\mathcal{H}}_n(S)\subset\Hom(\pi_1(S),\SL(n,\R))
\]
be the connected component of representations corresponding to the Hitchin component. If $\widetilde{\mathcal{H}}_n(S)_\Q$ denotes the set of representations whose image is contained in $\SL(n,\Q)$, then Theorem \ref{maintheorem}, along with the well-known fact that $\SL(n,\Q)$ is dense in $\SL(n,\R)$, gives rise to the following immediate corollary.

\begin{corollary}
\label{maincorollary}
$\widetilde{\mathcal{H}}_n(S)_\Q$ is dense in $\widetilde{\mathcal{H}}_n(S)$.
\end{corollary}

\subsubsection*{Acknowledgements} 
Both authors thank Darren Long and Arnaud Maret for helpful discussions and feedback regarding this work. In particular, we are indebted to the latter for explaining to us \cite{GoldmanXia09}*{Lemma 3.2} and its implications and relation to forthcoming work with Julien March\'e and Maxime Wolff on representations into Hilbert modular groups. These ideas were key in the beginning of this work. Both authors also thank the anonymous referee for helpful feedback and the second author acknowledges the support of the Big Bang Theory Graduate Fellowship.

\section{Twist flows on the Hitchin component}
Our proof of Theorem \ref{maintheorem} utilizes a particular deformation of representations which we will use to construct rational approximations to an arbitrary Hitchin representation. The purpose of this section is to describe the nature of these deformations. Our treatment follows the one given in \cite{Goldman86}, where they are interpreted as generalizations of the twist flow of a hyperbolic surface. Our interest in these deformations is in the arithmetic control at the level of representations they provide. This is similar to the perspective taken in \cite{LT}, where these same deformations are used to perform deformations with substantial arithmetic control as well.

Let $\Tr:\SL(n,\R)\to\R$ be the trace and for any $\gamma\in\pi_1(S)$, let $\Tr_\gamma:\Hit_n(S)\to\R$ denote the function
\[
\Tr_\gamma([\rho])=\Tr(\rho(\gamma)).
\]
Define $F:\SL(n,\R)\to\mathfrak{sl}(n,\R)$ as
\[
F(A)=A-\frac{\Tr(A)}{n}I_n.
\]
For any non trivial $\gamma\in\pi_1(S)$ which is freely homotopic in $S$ to a nonseparating simple closed curve, let $S\backslash\gamma$ denote compact the surface one gets by deleting a regular open neighborhood of $\gamma$ from $S$. Recall that $\pi_1(S)$ is an HNN-extension of $\pi_1(S\backslash\gamma)$. Define a flow on $\Hom(\pi_1(S),\SL(n,\R))$ by setting
\[
\Xi_\gamma^t(\rho)(\alpha)=\begin{cases}
    \rho(\alpha) & \textrm{if }\alpha\in\pi_1(S\backslash\gamma)\\
    \exp(t F(\rho(\gamma)))\rho(\alpha) & \textrm{if }i(\alpha,\gamma)=+1
\end{cases}
\]
where $i$ denotes the algebraic intersection number. This does define a new representation of $\pi_1(S)$ because $\exp(tF(\rho(\gamma)))$ centralizes $\rho(\gamma)$. 

\begin{definition}
The resulting flow, $\Xi_{\gamma}^t$, on $\Hom(\pi_1(S),\SL(n,\R))$ is called the \textbf{generalized twist flow} about $\gamma$.
\end{definition}

The main result of this section is the following, which states that two Hitchin representations may be connected via a path which is a piecewise concatenation of twist flows of the above form.

\begin{lemma}
\label{bendconnection}
Let $\rho_1$ and $\rho_2$ be Hitchin representations. Then there exist nonseparating simple closed curves $\gamma_1,\ldots,\gamma_k$ and real numbers $t_1,\ldots,t_k$ so that the representation
\[
\rho_2':=\Xi_{\gamma_k}^{t_k}(\ldots(\Xi_{\gamma_1}^{t_1}(\rho_1))\ldots)
\]
is conjugate to $\rho_2$. In other words, $[\rho_2']=[\rho_2]$ on $\Hit_n(S)$.
\end{lemma}
\begin{remark}
Analogs of this lemma in the context of other Lie groups have been known before (e.g. \cite{GoldmanXia09}*{Lemma 3.2} proves this for $\operatorname{SU}(2)$). To our knowledge, a proof in the context of the $\SL(n,\R)$-Hitchin component has never been recorded, so we include one here.
\end{remark}

To establish this result, we first exploit a connection between the flows $\Xi_\gamma^t$ and the underlying geometry of the Hitchin component. In \cite{GoldmanSymplecticForm}, Goldman defines a symplectic form on the the character variety of a surface group that gives $\Hit_n(S)$ the structure of a connected symplectic manifold. Denote by $\xi_\gamma^t$ the flow associated to the Hamiltonian vector field of the function $\Tr_\gamma$. This flow is related to our earlier flow $\Xi_\gamma^t$ via the following result.

\begin{theorem}[\cite{Goldman86}*{Theorem 4.7}]
\label{flowscover}
The flow $\Xi_\gamma^t$ on $\widetilde{\Hit}_n(S)$ covers the flow $\xi_\gamma^t$ on $\Hit_n(S)$. 
\end{theorem}

Thus, to establish the result of Lemma \ref{bendconnection}, we analyze the action of the flows $\xi_\gamma^t$ on $\Hit_n(S)$. We do so via an application of the following theorem of Bridgeman, Canary and Labourie, which is regarded as a sort of ``infinitesimal marked trace rigidity'' for Hitchin representations.

\begin{theorem}[\cite{BCL}*{Proposition 10.1}]
\label{tracerigidity}
For any $[\rho]\in\Hit_n(S)$ the collection of differentials
\[
\{ (d\Tr_\gamma)_{[\rho]}\,:\, \gamma\textrm{ is a nonseparating simple closed curve}\}
\]
generates the cotangent space to $\Hit_n(S)$ at $[\rho]$.
\end{theorem}

It is worth noting that the use of this theorem in proving Lemma \ref{bendconnection} is the only step in our proof of Theorem \ref{maintheorem} where we need the fact that the genus of $S$ is at least $3$. More discussion on what is missing in the genus $2$ case is done in Section \ref{section:generalizations}.

\begin{proof}[Proof of Lemma \ref{bendconnection}]
Let $\mathfrak{G}$ denote the group generated by the flows $\xi_\gamma^t$ for all nonseparating simple closed curves $\gamma$ and all $t$. By Theorem \ref{tracerigidity}, this group acts transitively on $\Hit_n(S)$ (cf. \cite{GoldmanXia09}*{Lemma 3.2}). Thus, given $[\rho_1],[\rho_2]\in\Hit_n(S)$, there exist nonseparating simple closed curves, $\gamma_1,\ldots,\gamma_k$, and $t_1,\ldots,t_k\in\R$ so that 
\[
[\rho_2]=\xi_{\gamma_k}^{t_k}(\ldots(\xi_{\gamma_1}^{t_1}([\rho_1]))\ldots).
\]
Lemma \ref{bendconnection} then follows from the above equality viewed at the level of representations and Theorem \ref{flowscover}.
\end{proof}

\section{Proof of Theorem \ref{maintheorem}}
We now explain the steps to establishing our main result. The results of the previous section allow one to build approximations to Hitchin representations via twist flows which are essentially controlled by choices of matrices in certain centralizers. One may build rational approximations to these flows through an application of the following theorem to these matrix centralizes.

\begin{theorem}[\cite{platonovrapinchuk}*{Theorem 7.7}]
\label{weakapproximation}
For $G$ a connected algebraic group defined over $\Q$, $G(\Q)$ is dense in $G(\R)$ endowed with the Euclidean topology.
\end{theorem}

For $A\in\SL(n,\Q)$, let $Z_A$ denote the centralizer of $A$ in the algebraic group $\SL_n$. It is an algebraic group defined over $\Q$ and for any subfield $k$ of $\C$, $Z_A(k)$ is the set of matrices in $\SL(n,k)$ that commutes with $A$. If $A\in\SL(n,\Q)$ has distinct eigenvalues then $Z_A$ is a connected algebraic group. Indeed, we have $Z_A(\C)\cong (\C^{*})^{n-1}$.

\begin{corollary}
\label{corollarydensity}
    If $A\in\SL(n,\Q)$ has distinct eigenvalues then $Z_A(\Q)$ is dense in $Z_A(\R)$.
\end{corollary}

This corollary will allow us to approximate the matrices controlling the twist deformations by rational ones, which we can then leverage for extra control on the arithmetic of the individual representations. Next, we show that $\mathcal{H}_n(S)_{\Q}$ is nonempty.

\begin{lemma}
\label{lemmaexistenceofQpoints}
For every $n\geq 2$, there exists a Hitchin representation with image in $\SL(n,\Q)$.
\end{lemma}
\begin{proof}
Observe that there exist representations inside the Teichm\"uller space of $S$ with image contained in $\SL(2,\Q)$. There are a number of constructions of such examples, such as ones due to Vinberg in \cite{vinberg_someexamplesSL2Q}, or from Takeuchi's result in \cite{takeuchi}. One might also consider the following example originally due to Long and Reid, based on work of Magnus in \cite{magnus}:
\[ 
\rho_0(a)=\begin{pmatrix} 3 & \frac{2}{3}
                        \\ 0 & \frac{1}{3}
                        \end{pmatrix}
\quad
\textrm{and}
\quad
\rho_0(b)=\begin{pmatrix} 0 & -2\\
                \frac{1}{2} & \frac{83}{8}
                \end{pmatrix}.
\] 
This is a discrete and faithful representation of the group $\Gamma=\langle a,b\,|\,[a,b]^2\rangle$ into $\SL(2,\Q)$. $\Gamma$ contains an index $4$ subgroup isomorphic to the fundamental group of a genus $2$ surface, hence contains finite-index surface subgroups of every genus. The restriction of $\rho_0$ to one such subgroup isomorphic to $\pi_1(S)$ gives a representation in the Teichm\"uller space of $S$ with image in $\SL(2,\Q)$. Using the irreducible embedding of $\SL(2,\Q)$ into $\SL(n,\Q)$, the conclusion holds for every $n$.
\end{proof}

We can now prove the main result.

\begin{proof}[Proof of Theorem \ref{maintheorem}]
Fix a representation $\rho_0:\pi_1(S)\to\SL(n,\Q)$ coming from Lemma \ref{lemmaexistenceofQpoints} and for each $k\geq 1$, set
\[
\Hit_n^k(S)=\left\{[\rho]\in\Hit_n(S)\,:\, 
\begin{tabular}{c}
there exist simple, nonseparating $\gamma_1,\ldots,\gamma_k\in\pi_1(S)$\\
and $t_1,\ldots,t_k\in\R$ so that $[\rho]=\xi_{\gamma_k}^{t_k}(\ldots(\xi_{\gamma_1}^{t_1}([\rho_0]))\ldots)$
\end{tabular}
\right\}.
\]
By Lemma \ref{bendconnection}, $\Hit_n(S)=\bigcup_{k\geq 1}\Hit_n^k(S)$, thus to show the result, it suffices to show that the closure of $\Hit_n(S)_\Q$ contains $\Hit_n^k(S)$ for all $k$.

For any $[\rho]\in\Hit_n^1(S)$, $\rho$ is conjugate to $\Xi_{\gamma_1}^{t_1}(\rho_0)$ for some nonseparating $\gamma_1$ and $t_1\in\R$. By Corollary \ref{corollarydensity}, we may take a sequence $\{B_j\}_j$ of elements of $Z_{\rho_0(\gamma_1)}(\Q)$ converging to $\exp(t_1F(\rho_0(\gamma_1)))$. Since $\pi_1(S)$ is an HNN-extension of $\pi_1(S\backslash\gamma_1)$, we may define a sequence of representations $\{\rho_j:\pi_1(S)\to\SL(n,\Q)\}_j$ by
\[
\rho_j(\alpha)=\begin{cases}
\rho_0(\alpha) & \textrm{if }\alpha\in\pi_1(S\backslash\gamma_1)\\
B_j\rho_0(\alpha) & \textrm{if }i(\alpha,\gamma_k)=+1.
\end{cases}
\]
The sequence $\{\rho_j\}_j$ converges to $\Xi_{\gamma_1}^{t_1}(\rho_0)$, hence the closure of $\Hit_n(S)_\Q$ contains $\Hit_n^1(S)$.

Now, suppose that the closure of $\mathcal{H}_n(S)_{\Q}$ contains $\mathcal{H}^{k-1}_n(S)$ for some $k\geq2$ and take $[\rho]\in\Hit_n^{k}(S)$. Then there exists a nonseparating $\gamma_k$ and $t_k\in\R$ so that $\rho$ is conjugate to $\Xi_{\gamma_k}^{t_k}(\sigma)$ for some $[\sigma]\in\Hit_n^{k-1}(S)$. Since $[\sigma]\in\Hit_n^{k-1}(S)$, we may let $\{\sigma_i:\pi_1(S)\to\SL(n,\Q)\}_i$ be a sequence of Hitchin representations with $[\sigma_i]\to[\sigma]$. By applying a conjugation by elements of $\SL(n,\Q)$, we may further assume that $\sigma_i\to\sigma$. 

For each $i$, let $\{B_{i,j}\}_j$ be a sequence in $Z_{\sigma_i(\gamma_k)}(\Q)$ converging to $\exp(t_k F(\sigma_i(\gamma_k)))$ in $Z_{\sigma_i(\gamma_k)}(\R)$, as given by Corollary \ref{corollarydensity}. Fix a distance $d$ on $\SL(n,\R)$ inducing its usual topology and for each $m\geq 1$, let $\phi(m)$ denote the smallest $i$ such that
\[
d(\exp(t_k F(\sigma_i(\gamma_k))),\exp(t_k F(\sigma(\gamma_k))))<\frac{1}{m}.
\]
Similarly, define $\psi(m)$ as the smallest $j$ such that
\[
d(B_{\phi(m),j},\exp(t_k F(\sigma_{\phi(m)}(\gamma_k))))<\frac{1}{m}.
\]
By construction of $\phi$ and $\psi$, $B_{\phi(m),\psi(m)}$ converges to $\exp(t_k F(\sigma(\gamma_k)))$ as $m\to\infty$. We then define a new sequence of representations $\{\rho_m:\pi_1(S)\to\SL(n,\R)\}_m$ by noting that $\pi_1(S)$ is an HNN-extension of $\pi_1(S\backslash\gamma_k)$ and setting
\[
\rho_m(\alpha)=\begin{cases}
    \sigma_{\phi(m)}(\alpha) & \textrm{if }\alpha\in\pi_1(S\backslash\gamma_k)\\
    B_{\phi(m),\psi(m)}\sigma_{\phi(m)}(\alpha) & \textrm{if }i(\alpha,\gamma_k)=+1.
\end{cases}
\]
By construction, $\rho_m(\pi_1(S))\leqslant\SL(n,\Q)$ for all $m$. As $\sigma_{\phi(m)}\to\sigma$ and $B_{\phi(m),\psi(m)}\to\exp(t_k F(\sigma(\gamma_k)))$, we see that $\rho_m\to\Xi_{\gamma_k}^{t_k}(\sigma)$ as $m\to\infty$. In particular, $\{[\rho_m]\}_m$ is a sequence in $\Hit_n(S)_\Q$ converging to $[\rho]$, so that the closure of $\Hit_n(S)_\Q$ contains $\Hit_n^k(S)$.
\end{proof}

\section{Generalizations of Theorem \ref{maintheorem}}
\label{section:generalizations}
We now discuss possible generalizations of Theorem \ref{maintheorem}.

\subsection{Changing the target group}

Let $F$ be a field. We fix a non-degenerated alternating bilinear form on $F^{2k}$ and we denote by $\Sp(2k,F)$ the subgroup of $\SL(2k,F)$ that preserve this form. Up to conjugation, this is independant of the chosen bilinear form. We also denote by $\textrm{G}_2(F)$ the subgroup of $\SL(7,F)$ defined in Definition 3.6 in \cite{audibert2022zariskidense}. The real Lie groups $\Sp(2k,\R)$ and $\textrm{G}_2(\R)$ are split.

Given a split linear real Lie group $G$, the character variety $$\textrm{Hom}(\pi_1(S),G)/G$$ admits a smooth connected component called the \emph{Hitchin component} of $G$ \cite{Hitchin_LiegroupsandTeichmullerspace}. We denote it by $\Hit_G(S)$.
For $G=\Sp(2k,\R)$ (resp. $G=\textrm{G}_2(\R)$), the Hitchin component of $G$ embeds into $\mathcal{H}_{2k}(S)$ via the natural inclusion $G\hookrightarrow\SL(2k,\R)$ (resp. $\Hit_7(S)$ via $\operatorname{G}_2(\R)\hookrightarrow\SL(7,\R)$) \cite{sambarino2023infinitesimalzariskiclosurespositive}.

\begin{theorem}
\label{GHitchindense}
    When the genus of $S$ is at least $3$, the set of equivalence classes of Hitchin representations with image in $\Sp(2k,\Q)$ \emph{(}resp. {$\operatorname{G}_2(\Q)$}\emph{)} is dense in $\Hit_{\Sp(2k,\R)}(S)$ \emph{(}resp. {$\Hit_{\operatorname{G}_2(\R)}(S)$}\emph{)}.
\end{theorem}

\begin{proof}
The proof of Theorem \ref{maintheorem} applies without change. Let $G=\Sp(2k,\R)$ and $n=2k$ or $G=\operatorname{G}_2(\R)$ and $n=7$. The Hitchin component of $G$ is a smooth submanifold of $\Hit_n(S)$.\footnote{It follows from the fact that any representation in $G$ which is reductive and which has finite centralizer is a smooth point of the character variety \cite{GoldmanSymplecticForm}; this is the case for every Hitchin representation \cite{sambarino2023infinitesimalzariskiclosurespositive}.} For any $\gamma\in\pi_1(S)$, denote by $\Tr^G_{\gamma}$ the restriction of $\Tr_{\gamma}$ to $\Hit_G(S)$. At any $[\rho]\in\Hit_G(S)$,
\[
\{(d\Tr^G_{\gamma})_{[\rho]}:\gamma\ \textrm{is a nonseparating simple closed curve}\}
\]
generates $T^*_{[\rho]}\Hit_G(S)$. Indeed, given $v\in T_{[\rho]}\Hit_G(S)$ such that $(d\Tr^G_{\gamma})_{[\rho]}(v)=0$ for all nonseparating simple closed curves $\gamma$, viewing $v$ as an element of $T_{[\rho]}\Hit_n(S)$, we have that $(d\Tr_{\gamma})_{[\rho]}(v)=0$ for all such $\gamma$, hence by Theorem \ref{tracerigidity}, $v=0$. As in the proof of Lemma \ref{bendconnection}, it follows that any point in $\Hit_G(S)$ is the image of any other point under successive application of twist flows along nonseparating simple closed curves.

For $A\in\Sp(2k,\Q)$ (resp. $\operatorname{G}_2(\Q)$), let $Z_A$ denote the centralizer of $A$ in the algebraic group $\Sp_{2k}$ (resp. $\operatorname{G}_2$). When $A$ has distinct eigenvalues (such as when $A=\rho(\gamma)$ for $\rho$ Hitchin), $Z_A(\Q)$ remains dense in $Z_A(\R)$. This follows from Theorem \ref{weakapproximation}, where $Z_A$ is a connected algebraic group by \cite{springersteinberg}*{\S3.9}.

The last thing to check is that there exist Hitchin representations into $\Sp(2n,\Q)$ (resp. $\operatorname{G}_2(\Q)$). This follows from the fact that the irreducible representation of $\SL(2,\R)$ into $\SL(n,\R)$ for $n=2k$ (resp. $n=7$) maps $\SL(2,\Q)$ into $\Sp(2k,\Q)$ (resp. $\operatorname{G}_2(\Q)$) since it can be written using only polynomials with coefficients in $\Z$. Composing it with a representation in the Teichmüller space of $S$ with image in $\SL(2,\Q)$ yields a Hitchin representation in $\Sp(2k,\Q)$ (resp. $\operatorname{G}_2(\Q)$). The proof now strictly follows the one of Theorem \ref{maintheorem}.
\end{proof}

Theorem \ref{GHitchindense} can also be extended to other $\Q$-groups, such as ones whose $\R$-points are isomorphic to $\SO(k+1,k)$ when $n=2k+1$, provided these groups contain the image of a Hitchin representation. It turns out that some of these groups do not contain the image of any Hitchin representations which factor through a principal embedding of $\SL(2,\R)$ (see \cite{audibert2022zariskidense} for the classification of $\Q$-groups which contain rational Hitchin representations along the Fuchsian locus). Thus establishing the existence of one such representation is not immediate in these cases.

\subsection{Further questions}
Many of the methods used in establishing these results apply in more general contexts than just the Hitchin component. Let $G$ be a reductive algebraic group defined over $\Q$, and $\Hom(\pi_1(S),G(\R))/ G(\R)$ the $G(\R)$-character variety of $S$. In general, this space is highly singular, but there is a Zariski-open subset $\Omega\subset\Hom(\pi_1(S),G(\R))$ so that $\Omega/G(\R)$ is a smooth manifold. Goldman's symplectic form in \cite{GoldmanSymplecticForm} is still defined on $\Omega/G(\R)$, giving it the structure of a symplectic manifold. Moreover, invariant functions associated to simple closed curves still induce flows which admit concrete descriptions in terms of HNN-extensions or free products with amalgamation as before \cite{Goldman86}, allowing one to perform deformations of representations in a manner that makes their arithmetic properties transparent. In light of these facts being key in the proof of Theorem \ref{maintheorem}, it is natural to ask to what extent our methods might establish similar results outside the context of the Hitchin component.

More specifically, let $X\subset\Omega/G(\R)$ be a connected component and denote by $X_\Q$ the set of representations in $X$ conjugate into $G(\Q)$. Is it, in fact, the case that $X_\Q$ is dense in $X$? The following two questions are necessary considerations in extending the methods of this work to a more general context.

\begin{question}
\label{question:nonempty}
Do there exist representations in $X_\Q$?
\end{question}

The first obstacle requires establishing the existence of an initial representation to start the deformation process. Though perhaps simpler of the two questions, this one can be surprisingly subtle. For instance, one may adopt constructions in \cite{goldmancomponents} to show that representations of closed surface groups into $\operatorname{PSL}(2,\Q)$ of every Euler number exist, though as the remarks of the previous section indicate, for arbitrary $G$, there may be no representations of a closed surface group into $G(\Q)$ which factor through an embedding of $\operatorname{PSL}(2,\R)$ into $G(\R)$. Thus more understanding for when $X_\Q$ is empty or not, and generalizations of the constructions used in Lemma \ref{lemmaexistenceofQpoints}, are needed.

The second missing condition is an infinitesimal one. To any conjugation invariant, differentiable $f:G(\R)\to\R$ and $\gamma\in\pi_1(S)$, one can form the function $f_\gamma:X\to\R$ by taking $f_\gamma([\rho])=f(\rho(\gamma))$. If $\mathcal{S}\subset\pi_1(S)$ denotes the collection of elements corresponding to simple closed curves on $S$, then one needs to consider the following.

\begin{question}
\label{question:infinitesimalrigidity}
Is there a collection of conjugation invariant functions $\mathcal{F}=\{f:G(\R)\to\R\}$ so that the differentials
\[
\{d f_\gamma\,:\, f\in\mathcal{F},\gamma\in\mathcal{S}\}
\]
span the cotangent space $T_{[\rho]}^*X$ at every $[\rho]\in X$?
\end{question}

Theorem \ref{tracerigidity} shows that when $X=\Hit_n(S)$ and the genus of $S$ is $3$ or more, then $\mathcal{F}=\{\Tr\}$ suffices, but the proof of this result in \cite{BCL} relies on a certain configuration of simple closed curves which can only exist when the genus of $S$ is at least $3$. We still expect that the result for the Hitchin component is true in genus $2$, noting that one is allowed to consider more general classes of functions than just the trace.

Question \ref{question:infinitesimalrigidity} thus asks to what extent the infinitesimal simple closed curve rigidity of Theorem \ref{tracerigidity} might hold in the whole character variety. Work of March\'e and Wolff shows that the analogous \textit{non-infinitesimal} rigidity statement can fail in the $\operatorname{PSL}(2,\R)$-character variety outside of Teichm\"uller space, but do not rule out whether the infinitesimal version might still hold \cite{marchewolff}. Goldman and Xia's work in \cite{GoldmanXia09} also answers this question positively for the whole $\operatorname{SU}(2)$-character variety, where doing so is a key step in establishing the ergodicity of the action of the mapping class group, illustrating how this question can arise in a number of natural contexts even outside the scope of this work.

Nonetheless, an affirmative answer to Question \ref{question:infinitesimalrigidity} for $X$ implies that the Hamiltonian flows associated to the functions $f_\gamma$ will still act transitively on $X$. By virtue of the $\gamma$ coming from simple closed curves, these flows still admit descriptions as generalized twist flows at the level of representations, allowing one to connect any two representations by a finite sequence of these deformations. Connecting an arbitrary such representation to one provided by a positive answer to Question \ref{question:nonempty}, along with further understanding of the arithmetic of the centralizers associated to simple closed curves indicate possible means one may generalize the work here to a broader context.
\bibliographystyle{amsalpha}
\bibliography{references}
\end{document}